\documentclass[11pt,a4paper,reqno]{amsart}
\usepackage[utf8]{inputenc}
\usepackage{mystyle}
\usepackage[foot]{amsaddr}

\newcommand{\By}[2]{\overset{\mbox{\tiny{#1}}}{#2}}
\newcommand{\ByRef}[2]{   \By{\eqref{#1}}{#2} }

\newcommand{\leByRef}[1]{ \ByRef{#1}{\le} }
\newcommand{\geByRef}[1]{ \ByRef{#1}{\ge} }

\def\firstauth#1{}
\def\second#1{}
\def\third#1{}
\begin{document}

\title{On cospectral graphons}

\author[Hladk\'y]{Jan Hladk\'y}
\address{Institute of Computer Science of the Czech Academy of Sciences, Pod Vod\'{a}renskou v\v{e}\v{z}\'{\i} 2, 182~07 Prague, Czechia. With institutional support RVO:67985807.}
\email{hladky@cs.cas.cz}
\thanks{(JH): Research supported by Czech Science Foundation Project 21-21762X.}

\author[I\v{l}kovi\v{c}]{Daniel I\v{l}kovi\v{c}}
\address{Faculty of Informatics, Masaryk University, Botanick\'a 68A, 602 00 Brno, Czech Republic.}
\email{493343@mail.muni.cz}
\thanks{(DI, XS): Research supported by the MUNI Award in Science and Humanities (MUNI/I/1677/2018) of the Grant Agency of Masaryk University.}

\author[León]{Jared León}
\email{jared.leon@warwick.ac.uk}
\address{Mathematics Institute, University of Warwick, Coventry, UK}
\thanks{(JL): Research supported by the Warwick Mathematics Institute Centre for Doctoral Training.}
  
\author[Shu]{Xichao Shu}
\email{540987@mail.muni.cz}


\begin{abstract}
In this short note, we introduce cospectral graphons, paralleling the notion of cospectral graphs. As in the graph case, we give three equivalent definitions: by equality of spectra, by equality of cycle densities, and by a unitary transformation. We also give an example of two cospectral graphons that cannot be approximated by two sequences of cospectral graphs in the cut distance.
\end{abstract}

\maketitle

\section{Introduction}\label{sec:Intro}
There are several prominent equivalences defined on finite graphs on vertex set $[n]$ which have the following form for graphs $G$ and $H$,
\begin{enumerate}[(F1)]
    \item\label{F1} For the adjacency matrix $A_G$ of $G$ and for the adjacency matrix $A_H$ of $H$ we have $TA_H=A_GT$ for a suitable $n\times n$ matrix $T\in\mathcal{M}_n$, and equivalently
    \item\label{F2} We have the equality of homomorphism counts $\hom(F,G)=\hom(F,H)$ for all $F\in\mathcal{F}$.\footnote{Recall that a \emph{homomorphism of a graph $F$ to a graph $G$} is any map $h:V(F)\to V(G)$ such that $h(u)h(v)\in E(G)$ for every $uv\in E(F)$. The symbol $\hom(F,G)$ denotes the size of the set of all homomorphisms.}
\end{enumerate}
These equivalencies are graph isomorphism, fractional isomorphism, cospectrality, and quantum isomorphism. Let us give details and specify the families of matrices $\mathcal{M}_n$ and of graphs $\mathcal{F}$.
\begin{itemize}
    \item The most important of such equivalences is the usual \emph{graph isomorphism}. In that case, obviously, $\mathcal{M}_n$ is the set of all $n\times n$ permutation matrices. The characterization of graph isomorphism using~\ref{F2} is a famous result of Lov\'asz~\cite{MR214529}, that is, for  $\mathcal{F}=\{\text{all graphs}\}$, the test~\ref{F2} passes only for $G$ and $H$ isomorphic.
\item \emph{Fractional isomorphism} concept was introduced in~\cite{MR0843938} using~\ref{F1} for $\mathcal{M}_n$ equal to all doubly stochastic matrices. Later, several further equivalent definitions were given in~\cite{MR1297385}. Finally, Dvořák~\cite{MR2668548} gave a characterization as in~\ref{F2}, with $\mathcal{F}$ being all trees. 
\item \emph{Cospectrality} (sometimes called \emph{isospectrality}) is derived from linear algebra, where this notion denotes equality of spectra (including multiplicities) of matrices. For graphs, this concept is applied to adjacency matrices.\footnote{Less often, some other matrices, such as the normalized Laplacian, are used. These alternative definitions lead to different notions of cospectrality.} Basic linear algebra then asserts that for a characterization using~\ref{F1} we need to take $\mathcal{M}_n$ to be all orthonormal $n\times n$ matrices. For a characterization using~\ref{F2} we take $\mathcal{F}$ to be all cycles. The key insight to infer the equivalence of these two definitions is that when we write $C_k$ for the cycle of length $k$, $\hom(C_k,G)$ is equal to the trace of $A_G^k$, which is in turn equal to sum of the $k$-th powers of the eigenvalues.
\item \emph{Quantum isomorphism} is the newest addition to the list. It was introduced in~\cite{ATSERIAS2019289}, using two definitions, of which we mention only the one in the form~\ref{F1}. The class of matrices $\mathcal{M}_n$ is in a sense similar to that used in fractional isomorphism. Indeed, it is also required the entries in each row and column sum to the identity. The difference is that the entries are now not reals in $[0,1]$ but rather orthogonal projectors (say, in an $n$-dimensional complex vector space). In~\cite{quantumisomorphismequivalentequality} it was shown that this definition is equivalent to~\ref{F2} with $\mathcal{F}$ being all planar graphs.
\end{itemize}

In this note we translate the concept of cospectrality to graphons. Graphons are analytic objects introduced by Borgs, Chayes, Lovász, Sós, Szegedy, and Vesztergombi~\cite{LOVASZ2006933,BORGS20081801} to represent limits of sequences of graphs. We recall the definition and properties we shall need in Section~\ref{sec:Preliminaries}. At this moment, it suffices that each graphon is a symmetric bounded Lebesgue measurable function on $[0,1]^2$. In particular, it can be viewed as an integral kernel operator on $L^2([0,1])$. There are several directions in which graphons can be studied. Below are the two most important ones.
\begin{itemize}
    \item Graphons can be studied \emph{per se}, often motivated by concepts from finite graphs. As one of the most prominent examples, the seminal paper~\cite{LOVASZ2006933} introduces the notion of \emph{homomorphism density} $t(F,W)$ (here, $F$ is a graph and $W$ is a graphon) which is motivated by the notion of homomorphism counts $\hom(F,G)$ (or, perhaps better, by the notion of homomorphism density $\hom(F,G)v(G)^{-v(F)}$. Other concepts include for example notions of max-cut~\cite{BORGS20081801}, or the chromatic number~\cite{HlaRoc:IndependentSets}. Our attention is on~\ref{F1} and~\ref{F2}. Note that there is a very easy way to introduce the counterparts of graph isomorphism, fractional isomorphism, cospectrality, or quantum isomorphism for graphons by the above characterizations using~\ref{F2}. Namely, we will say that two graphons $U$ and $W$ are (i)~\emph{weakly isomorphic}, (ii)~\emph{fractionally isomorphic}, (iii)~\emph{cospectral}, (iv)~\emph{quantum isomorphic} if $t(F,U)=t(F,W)$ for (i)~every graph~$F$, (ii)~every tree~$F$, (iii)~every cycle~$F$, (iv)~every planar graph~$F$, respectively. It is now of interest to extend any further properties of these equivalence relations to graphons. As for weak isomorphism, the main result of~\cite{BoChLo:Moments} asserts that two graphons $U$ and $W$ are weakly isomorphic if and only if for every $\eps>0$ there exists a unitary Koopman operator $T$ so that the operator norm of $TU-WT$ is less than $\eps$. Recall that being a unitary Koopman operator amounts to the existence of a measure-preserving bijection $\pi:[0,1]\to[0,1]$ such that $(Tf)(x)=f(\pi(x))$ for every $f\in L^2([0,1])$ and $x\in[0,1]$. This is a clear counterpart to~\ref{F1}.\footnote{Note that we cannot achieve precisely $TU=WT$ in some situations, see Figure~7.1 in~\cite{MR3012035}.} The theory of fractional isomorphism was introduced by Grebík and Rocha,~\cite{Grebik2022}. Specifically, \cite{Grebik2022} gives graphon counterparts to all known characterizations of graph fractional isomorphism and proves their equivalence to the above definition using tree densities. The characterization of~\ref{F1} is $TU=WT$ for a suitable `Markov operator' $T$. One of the main contributions in this note is that in Theorem~\ref{thm:spec_equiv} we give several definitions of `cospectral graphons' and prove their equivalence. As far as we know, no attempt to extend the notion of quantum isomorphism to graphons has been made.
    \item A particularly important line of research is in investigating continuity properties of various graph(on) parameters. The most prominent example is again the homomorphism density. In particular, the so-called Counting lemma asserts that if $(G_n)_n$ is a sequence of graphs converging to a graphon $W$, then for every graph $F$ we have $\hom(F,G_n)v(G_n)^{-v(F)}\to t(F,W)$. This yields positive results for the for equivalences above. Namely, suppose that $(G_n)_n$ and $(H_n)_n$ are sequences of graphs so that for each $n$, the graphs $G_n$ and $H_n$ are of the same order and isomorphic, or fractionally isomorphic, or cospectral, or quantum isomorphic. Suppose further that $(G_n)_n$ converges to a graphon $U$ and $(H_n)_n$ converges to a graphon $W$. Then by the above, $G_n$ and $H_n$ have the same counts (and thus also densities) of homomorphisms from each graph, or from each tree, or from each cycle, or from each planar graph, respectively. By the continuity of homomorphism densities, these equalities are inherited to $U$ and $W$. We conclude that $U$ and $W$ are weakly isomorphic, or fractionally isomorphic, or cospectral, or quantum isomorphic. The main result of~\cite{HlaHng:ApproximatingFIG} goes in the opposite direction in the case of fractional isomorphism: If $U$ and $W$ are fractionally isomorphic then we can find sequences of graphs $G_n\to U$ and $H_n\to W$ such that for each $n$, the graphs $G_n$ and $H_n$ are fractionally isomorphic. The second result of this note, Theorem~\ref{thm:cospectralinaproximable}, is that the counterpart of this result for cospectrality does not hold.
\end{itemize}

\section{Preliminaries}\label{sec:Preliminaries}
For a graph $G$ we denote by $v(G)$ the number of vertices of $G$ and $e(G)$ the number of edges. By $C_k$ we denote cycle of length $k$.

%\section{Cospectral graphs}

%Two non-isomorphic graphs are considered \emph{cospectral} concerning a given matrix if they have the same eigenvalues. Cospectral graphs help to show the limitations that the spectrum of a particular matrix might have in distinguishing the properties of a graph. Several different matrices are used in spectral graph theory, which can reveal different information about a graph. So graphs may be cospectral concerning some matrix but not cospectral concerning another matrix (though there are graphs which are cospectral concerning all matrices).

%Some of the common matrices that are studied include the adjacency matrix A, the combinatorial Laplacian $L := D - A$, the signless Laplacian $Q := D + A$, the normalized Laplacian $\mathcal{L}:= D^{-1/2}(D - A)D^{-1/2}$ (with the convention that if a vertex is isolated then the corresponding entry id $D^{-1/2}$ is 0), and the Seidel matrix $S := J - I - 2A$. The Seidel matrix is not as commonly studied and is defined by putting a -1 for each edge, a 1 for each non-edge, and a 0 on the diagonal entries.

%\section{Cospectral graphons}

%In this section, we generalize the notion of the cospectral graphs to cospectral graphons. A
We now recall basic of the theory of graphons, using mostly notation from the excellent monograph~\cite{MR3012035}. A \emph{graphon} is a symmetric measurable function $W : [0, 1]^2 \rightarrow [0,1]$. When we view the sets $[0,1]$ or $[0,1]^2$ as measure spaces, we take the underlying measure to be the Lebesgue measure. The \emph{cut-norm distance} is defined as
\begin{equation}\label{eq:defCutNorm}
d_\square(U, W) = \underset{S, T \subseteq [0,1]}{\sup}\left|\int_{S\times T} U(x, y) - W(x, y) \, dx\, dy\right|,
\end{equation}
where the supremum ranges over all measurable subsets $S, T$ of $[0,1]$. The \emph{cut-distance} is defined as 
\begin{equation*}
\delta_\square(U, W) = \inf_\varphi d_\square(U^\varphi, W),    
\end{equation*}
where infimum ranges over all measure preserving bijections $\varphi: [0,1] \rightarrow [0,1]$ and $U^\varphi (x,y) = U(\varphi(x), \varphi(y))$. Obviously, the cut-distance is in fact a pseudodistance.

Recall that each graph can be represented as a graphon. That is, if $G$ is a graph on vertex set $V$, then we define a graphon $W_G$ in the following way. Partition $[0,1]$ into sets $\{\Omega_v\}_{v\in V}$ of measure $\frac{1}{|V|}$ each. The representation $W_G$ is not unique as it depends on the choice of the partition $\{\Omega_v\}_{v\in V}$. However any two representations are at (pseudo-)distance~0 with respect to the cut distance. The phrase \emph{sequence of graphs $(G_n)_n$ converges to $W$} which appeared several times above then means $\delta_\square(W_{G_n},W)\to 0$ and $v(G_n)\to\infty$.

The density of a graph $G$ in a graphon $W$ is defined as
\begin{equation*}
t(G, W) = \int_{(x_v)_{v\in v(G)} \in [0,1]^{V(G)}} \underset{uv\in E(G)}{\prod} W(x_u,x_v)\prod_{v\in V(G)}dx_v\;.
\end{equation*}

So far, the analogies between the adjacency matrix of a graph and a graphon were made in the original domain $[0,1]^2$. In the language of Fourier transform, this would correspond to the spatial perspective. We now move to the dual perspective. Namely, we recall formalism that allows us to study spectral properties of graphons. Details can be found in Section~7.5 of~\cite{MR3012035}. Each graphon $W$ can be associated with an operator $T_W: L^2([0,1]) \rightarrow L^2([0,1])$ defined by 
\begin{equation*}
(T_W f)(x) = \int_0^1 W(x,y)f(y)\, dy\quad \mbox{for every $x\in[0,1]$}. 
\end{equation*}
It is known that $T_W$ is a symmetric real Hilbert-Schmidt operator. In particular, $f\in L^2([0,1])$ and $\lambda\in\mathbb{R}$ are \emph{eigenvector} and \emph{eigenvalue} of $W$ if $T_W f=\lambda f$. Since $T_W$ is a compact operator it has a discrete real spectrum (that is, the multiset of eigenvalues), denoted by $Spec(W)$. It is well-known that for each $\eps>0$ the number of eigenvalues of modulus at least $\eps$ (including multiplicities) is finite. Further, Parseval's Theorem (see e.g.~(7.23) in~\cite{MR3012035}) asserts that
\begin{equation}\label{eq:Parseval}
\|W\|_2^2=    \sum_{\lambda \in Spec(W)} \lambda^2\;.
\end{equation}

We mentioned in Section~\ref{sec:Intro} that for $k\ge 3$ and a graph $G$, the quantity $\hom(C_k,G)$ is the sum of the $k$-th powers of the eigenvalues of the adjacency matrix of $G$. The graphon counterpart to this is (see~(7.22) in~\cite{MR3012035})
\begin{equation}
\label{eq:eigen_cycles}
    t(C_k, W) =  \sum_{\lambda \in Spec(W)} \lambda^k\quad\mbox{ for each $k\ge 3$.}
\end{equation}

Last, recall that an operator $T$ on $L^2([0,1])$ is \emph{unitary} if $T$ is surjective and for every $f \in L^2([0,1])$, we have $\| T f \|_2 = \| f \|_2$. Unitary operators are functional-analytic counterparts to orthonormal matrices.

\section{Equivalent definitions of cospectral graphons}
We defined cospectral graphons in Section~\ref{sec:Intro}. Since not all the concepts were defined back then, we repeat the definition here.
\begin{definition}
    Graphons $U$ and $W$ are \emph{cospectral} if and only if for all integers $k\ge 3$, $t(C_k, U) = t(C_k, W)$.
\end{definition}

In our first main theorem, we prove that graphon cospectrality has several equivalent definitions.
\begin{theorem}
\label{thm:spec_equiv}
    For every two graphons $U$ and $W$, the following are equivalent: 
    \begin{enumerate}[(i)]
        \item \label{stm:1} For all integers $k$ bigger than $3$, we have  $t(C_k, U) = t(C_k, W)$.
        \item \label{stm:2} There are infinitely many odd numbers $k$ and infinitely many even numbers $k$, such that $t(C_k, U) = t(C_k, W)$.
        \item \label{stm:3} Graphons $U$ and $W$ have the same spectra, that is, $Spec(U) = Spec(W)$.
        \item \label{stm:4} There exists a unitary operator $T: L^2([0,1]) \rightarrow L^2([0,1])$, such that $T \circ T_W = T_U \circ T$.
    \end{enumerate}
\end{theorem}

To prove Theorem~\ref{thm:spec_equiv}, we shall prove $\ref{stm:1}\Rightarrow\ref{stm:2}\Rightarrow\ref{stm:3}\Rightarrow\ref{stm:1}$ and $\ref{stm:4}\Rightarrow\ref{stm:3}\Rightarrow\ref{stm:4}$. Implication $\ref{stm:1}\Rightarrow\ref{stm:2}$ holds trivially. Implication~$\ref{stm:3}\Rightarrow\ref{stm:1}$ follows directly from~\eqref{eq:eigen_cycles}.
     
 \begin{proof}[Proof of Theorem~\ref{thm:spec_equiv}, $\ref{stm:2}\Rightarrow\ref{stm:3}$]
Suppose that $Spec(U) \neq Spec(W)$. Let $\nu>0$ be the largest number such that there exist an eigenvalue of modulus $\nu$ with different multiplicity in $Spec(U)$ and in $Spec(W)$. Let $\alpha:=\sup\{|\lambda|:\lambda\in Spec(U)\cup Spec(W), |\lambda|<\nu\}$. We have $\alpha<\nu$. Further, let $\beta>0$ be small enough that
\begin{equation}\label{eq:zbytek}
\sum_{\lambda\in Spec(U),|\lambda|\le \beta}\lambda^2<\alpha^2    
\quad\mbox{and}\quad
\sum_{\lambda\in Spec(W),|\lambda|\le \beta}\lambda^2<\alpha^2    \;.
\end{equation}
Such a number $\beta$ exists since the eigenvalues are square-summable by~\eqref{eq:Parseval}.

Let $h_U=|\{\lambda\in Spec(U):|\lambda|\in (\beta,\nu)\}|$, $h_W=|\{\lambda\in Spec(W):|\lambda|\in (\beta,\nu)\}|$, and $h=h_U+h_W$.
Recall that the spectra of $U$ and $W$ are real and thus the only eigenvalues of modulus $\nu$ may be $\nu$ and $-\nu$. Let $m_U^+$ and $m_U^-$ be the multiplicities of $\nu$ and $-\nu$ in $Spec(U)$ and let $m_W^+$ and $m_W^-$ be the multiplicities of $\nu$ and $-\nu$ in $Spec(W)$. 
The key is the following observation, which follows for every $k\ge 3$ by substituting into~\eqref{eq:eigen_cycles},
\begin{align*}
t(C_k,U)=\underbrace{\sum_{\lambda\in Spec(U):|\lambda|>\nu} \lambda^k}_{\mathsf{(T1)}_U} + m_U^+ \nu^k + (-1)^k m_U^- \nu^k + \underbrace{\sum_{\lambda\in Spec(U):|\lambda|\in (\beta,\nu)} \lambda^k}_{\mathsf{(T2)}_U} + \underbrace{\sum_{\lambda\in Spec(U):|\lambda|\in (0,\beta)} \lambda^k}_{\mathsf{(T3)}_U}\;.
\end{align*}
We can write an analogous formula for $t(C_k,W)$, and get terms $\mathsf{(T1)}_W$, $\mathsf{(T2)}_W$, and $\mathsf{(T3)}_W$. We can express $|t(C_k,U)-t(C_k,W)|$, by grouping $\mathsf{(T1)}_U$ with $\mathsf{(T1)}_W$, then $\mathsf{(T2)}_U$ with $\mathsf{(T2)}_W$, and lastly $\mathsf{(T3)}_U$ with $\mathsf{(T3)}_W$. The terms $\mathsf{(T1)}_U$ and $\mathsf{(T1)}_W$ cancel perfectly due to the way we chose $\nu$. Next, we deal with $\mathsf{(T2)}$. Each summand contributes between $-\alpha^k$ and $\alpha^k$. There are $h_U$ summands in $\mathsf{(T2)}_U$. Similarly, there are $h_W$ summands in $\mathsf{(T2)}_W$. Hence, $|\mathsf{(T2)}_U-\mathsf{(T2)}_W| \le h\alpha^k$. Last, we deal with the terms $\mathsf{(T3)}$. We use the well-known inequality between the $\ell^p$-norm and the $\ell^q$-norm, $\|\cdot \|_p \ge \|\cdot\|_q$ for $p\le q$. For $p=2$ and $q=k$ this gives 
$$
|\mathsf{(T3)}_U|\le \left|\sum_{\lambda\in Spec(U):|\lambda|\in (0,\beta)} \lambda^k\right|\le \sum_{\lambda\in Spec(U):|\lambda|\in (0,\beta)} |\lambda|^k\le \left(\sum_{\lambda\in Spec(U):|\lambda|\in (0,\beta)} |\lambda|^2\right)^{k/2}\leByRef{eq:zbytek} \alpha^{k}\;,$$
and similarly $|\mathsf{(T3)}_W|\le \alpha^k$.
Putting all these bounds together, we conclude
\begin{equation}\label{eq:cycledifference}
    |t(C_k,U)-t(C_k,W)|=\left(m_U^+-m_W^+ + (-1)^k (m_U^- - m_W^-)\right)\nu^k \pm (h+2)\alpha^k\;.
\end{equation}
The point is that for $k$ large, $(h+2)\alpha^k$ is negligible with respect to $\nu^k$. Thus, to prove that for arbitrary large even $k$ or for arbitrary large odd $k$ we have $t(C_k,U)-t(C_k,W)\neq 0$, we only need to prove that for a suitable parity $P=0$ (corresponding to $k$ even) or $P=1$ (corresponding to $k$ odd) we have
\begin{equation}\label{eq:parity}
    m_U^+-m_W^+ + (-1)^P (m_U^- - m_W^-)\neq 0.
\end{equation}

We distinguish two cases. First, if we have $m_U^+ +m_U^-\neq m_W^+ +m_W^-$ then~\eqref{eq:parity} holds for $P=0$. Second, assume that we have $m_U^+ +m_U^-=m_W^+ +m_W^-$. In that case, $m_U^+ -m_U^-\neq m_W^+ -m_W^-$ (indeed, otherwise, we would have $m_U^+=m_W^+$ and $m_U^-=m_W^-$, a contradiction to the way we chose $\nu$). Thus, \eqref{eq:parity} holds for $P=1$.
\end{proof}

\begin{proof}[Proof of Theorem~\ref{thm:spec_equiv}, $\ref{stm:4}\Rightarrow\ref{stm:3}$]
 Since $T$ is a unitary operator, it is surjective and preserves the norm and scalar product. Hence, it also preserves the eigenvalues together with the dimensions of their eigenspaces. It follows that the multiplicities are preserved as well.
 \end{proof}

\begin{proof}[Proof of Theorem~\ref{thm:spec_equiv}, $\ref{stm:3}\Rightarrow\ref{stm:4}$]
The Spectral Theorem for compact self-adjoint operators tells us that we can decompose $L^2([0,1])$ according to the orthogonal eigenspaces of $T_U$,
\begin{equation*}
    L^2([0,1])=\bigoplus_{\lambda\in Spec(U)} K_\lambda\;.
\end{equation*}
Here, $\{K_\lambda\}_{\lambda\in Spec(U)}$ are mutually orthogonal spaces, each of dimension equal to the multiplicity of the eigenvector $\lambda$, and writing $P_K:    L^2([0,1])\to L^2([0,1])$ for the orthogonal projection on a closed subspace $K\subset L^2([0,1])$, we have $T_U=\sum_{\lambda\in Spec(U)}\lambda P_{K_\lambda}$. Likewise, we have $ L^2([0,1])=\bigoplus_{\lambda\in Spec(W)} L_\lambda$ and $T_W=\sum_{\lambda\in Spec(W)}\lambda P_{L_\lambda}$ for the eigenspaces of $T_W$. Since the spectra of $U$ and $W$ are the same including multiplicities, we can fix linear isometries $b_\lambda:L_\lambda\to K_\lambda$ for each $\lambda\in Spec(U)=Spec(W)$. It is clear that the operator $Tf:=\sum_{\lambda}b_\lambda\circ P_{L_\lambda}$ satisfies $T \circ T_W = T_U \circ T$. It is also clear that $T$ is surjective and preserves the $L^2$-norm.
\end{proof}

\section{Cospectral inapproximability}
As we mentioned in Section~\ref{sec:Intro}, the main motivation for our second main result is the following theorem of Hladký and Hng~\cite{HlaHng:ApproximatingFIG}.
\begin{theorem}
    Suppose that $U$ and $W$ are fractionally isomorphic graphons. Then there exist sequences  $(G_n)_n$ and $(H_n)_n$ of graphs such that $(G_n)_n$ converges to $U$, $(H_n)_n$ converges to $W$, and $H_n$ is $G_n$ are fractionally isomorphic for each $n$.
\end{theorem}
Here, we show that a counterpart of this theorem does not hold for cospectral graphons.
\begin{theorem}\label{thm:cospectralinaproximable}
    Consider graphons $U(x,y)=\frac{1}{2}$ and $W(x,y)=\mathbf{1}_{x\in [0,\frac12]}\cdot \mathbf{1}_{y\in [0,\frac12]}$ for $(x,y)\in[0,1]^2$. Then $U$ and $W$ are cospectral with $Spec(U)=Spec(W)=\{\frac12\}$.
    If we have sequences $(G_n)_n$ and $(H_n)_n$ of graphs such that $(G_n)_n$ converges to $U$, $(H_n)_n$ converges to $W$, then $G_n$ and $H_n$ are not cospectral for each $n$ sufficiently large.
\end{theorem}

%\begin{proposition}
%  \label{prop:cospec_seq_imply_cospectral}
%  Let~$U$ and~$W$ be graphons, and let~$U_n$ and~$W_n$ be sequences of graphons that
%  converge to~$U$ and~$W$ in cut-distance respectively. Then, if~$U_n$ and~$W_n$ are
%  cospectral for all~$n \geq 1$, the graphons~$U$ and~$W$ are cospectral.
%\end{proposition}
%\begin{proof}
%  By the counting lemma, for any graph~$F$, the sequence~$t(F, U_n)$ converges
%  to~$t(F, U)$. In particular, for an integer~$\ell \geq 2$, and by symmetry,
%  \begin{equation*}
%    t(C_\ell, U) = \lim_{n \to \infty} t(C_\ell, U_n) = \lim_{n \to \infty} t(C_\ell,
%    W_n) = t(C_\ell, W).
%  \end{equation*}
%  Then, graphons~$U$ and~$W$ are cospectral.
%\end{proof}

We now give a proof of Theorem~\ref{thm:cospectralinaproximable}. The statement about the spectra of $U$ and $W$ is easy to verify, with the only eigenvector of $U$ being the constant-1 and the only eigenvector of $W$ being $\mathbf{1}_{x\in [0,\frac12]}$. The particular property on which our proof depends is that $\|U\|_1\neq \|W\|_1$. That is, we prove that each two graphons with different $L^1$-norms (whether they are cospectral or not) cannot be cospectrally approximated, not only by finite graphs, but also by the more general class of $\{0,1\}$-valued graphons.
\begin{proposition}
  \label{prop:dist_norm_not_approx}
  Let~$U$ and~$W$ be graphons with~$\|U\|_1>\|W\|_1$. Suppose that $U'$ and $W'$ are two $\{0,1\}$-valued graphons with $\delta_\square(U,U'),\delta_\square(W,W')<(\|U\|_1-\|W\|_1)/2$. Then $U'$ are $W'$ not cospectral.
\end{proposition}
\begin{proof}
As a preparatory step, we claim that
\begin{equation}\label{eq:prerp}
\int U(x,y)-U'(x,y)dx dy\le \delta_\square(U,U')\;.
\end{equation}

To verify this, we need to check that for every measure preserving bijection $\varphi:[0,1]\to[0,1]$ we have
\begin{equation*}
\int U (x, y) -U'(x, y) dx dy=\int U^{\varphi} (x, y) -U'(x, y) dx dy\le d_\square(U^{\varphi},U')\;.
\end{equation*}
This becomes obvious when we consider $S=T=[0,1]$ in~\eqref{eq:defCutNorm}.

Thus,
\[
\|U'\|_1=\int U'(x,y)dx dy=\int U(x,y)dx dy-\int U(x,y)-U'(x,y)dx dy\geByRef{eq:prerp} \|U\|_1-\delta_\square(U,U')\;.
\]

Since $U'$ is $\{0,1\}$-valued, and since $0^2=0$ and $1^2=1$, we have $\|U'\|_2^2=\|U'\|_1$. We conclude that $\|U'\|_2^2\ge \|U\|_1-\delta_\square(U,U')$.

Similarly, $\|W'\|_2^2\le \|W\|_1+\delta_\square(W,W')$. Combining with the main assumption of the proposition, we get $\|U'\|_2^2\neq \|W'\|_2^2$. By Parseval's Theorem~\eqref{eq:Parseval} we conclude that $U'$ and $W'$ are not cospectral.
\end{proof}

We pose as an open problem, whether we can remove the assumption $\|U\|_1\neq \|W\|_1$.
\begin{problem}
  Do there exist two cospectral graphons $U$ and $W$ with $\|U\|_1=\|W\|_1$ that cannot be cospectrally approximated?
\end{problem}

\bibliographystyle{amsabbrv}
\bibliography{mybib}

\providecommand{\bysame}{\leavevmode\hbox to3em{\hrulefill}\thinspace}
\providecommand{\MR}{\relax\ifhmode\unskip\space\fi MR }
\providecommand{\MRhref}[2]{%
  \href{http://www.ams.org/mathscinet-getitem?mr=#1}{#2}
}
\providecommand{\href}[2]{#2}
\begin{thebibliography}{10}

\bibitem{ATSERIAS2019289}
A.~Atserias, L.~Mančinska, D.~E. Roberson, R.~Šámal, S.~Severini, and
  A.~Varvitsiotis, \emph{Quantum and non-signalling graph isomorphisms},
  Journal of Combinatorial Theory, Series B \textbf{136} (2019), 289--328.

\bibitem{BORGS20081801}
C.~Borgs, J.~Chayes, L.~Lovász, V.~Sós, and K.~Vesztergombi, \emph{Convergent
  sequences of dense graphs i: Subgraph frequencies, metric properties and
  testing}, Advances in Mathematics \textbf{219} (2008), 1801--1851.

\bibitem{BoChLo:Moments}
C.~Borgs, J.~Chayes, and L.~Lov\'asz, \emph{Moments of two-variable functions
  and the uniqueness of graph limits}, Geom. Funct. Anal. \textbf{19} (2010),
  1597--1619. \MR{2594615}

\bibitem{MR2668548}
Z.~Dvo\v{r}\'{a}k, \emph{On recognizing graphs by numbers of homomorphisms}, J.
  Graph Theory \textbf{64} (2010), 330--342. \MR{2668548}

\bibitem{Grebik2022}
J.~Greb{\'i}k and I.~Rocha, \emph{Fractional isomorphism of graphons},
  Combinatorica \textbf{42} (2022), 365--404.

\bibitem{HlaHng:ApproximatingFIG}
J.~Hladk\'y and E.~K. Hng, \emph{Approximating fractionally isomorphic
  graphons}, European J. Combin. \textbf{113} (2023), Paper No. 103751, 19.
  \MR{4604288}

\bibitem{HlaRoc:IndependentSets}
J.~Hladk\'y and I.~Rocha, \emph{Independent sets, cliques, and colorings in
  graphons}, European J. Combin. \textbf{88} (2020), 103108, 18. \MR{4111720}

\bibitem{MR214529}
L.~Lov\'asz, \emph{Operations with structures}, Acta Math. Acad. Sci. Hungar.
  \textbf{18} (1967), 321--328.

\bibitem{MR3012035}
\bysame, \emph{Large networks and graph limits}, American Mathematical Society
  Colloquium Publications, vol.~60, American Mathematical Society, Providence,
  RI, 2012. \MR{3012035}

\bibitem{lovász2012large}
L.~Lov{\'a}sz, \emph{Large networks and graph limits}, American Mathematical
  Society colloquium publications, American Mathematical Society, 2012.

\bibitem{LOVASZ2006933}
L.~Lovász and B.~Szegedy, \emph{Limits of dense graph sequences}, Journal of
  Combinatorial Theory, Series B \textbf{96} (2006), 933--957.

\bibitem{quantumisomorphismequivalentequality}
L.~Mančinska and D.~E. Roberson, \emph{Quantum isomorphism is equivalent to
  equality of homomorphism counts from planar graphs}, 2019.

\bibitem{MR1297385}
M.~V. Ramana, E.~R. Scheinerman, and D.~Ullman, \emph{Fractional isomorphism of
  graphs}, Discrete Math. \textbf{132} (1994), 247--265. \MR{1297385}

\bibitem{MR0843938}
G.~Tinhofer, \emph{Graph isomorphism and theorems of {B}irkhoff type},
  Computing \textbf{36} (1986), 285--300. \MR{843938}

\end{thebibliography}

\end{document}